\documentclass[11pt]{amsart}
\usepackage{amsmath}
\usepackage{bbold}
\usepackage{amsthm}
\usepackage{amsfonts}
\usepackage{graphicx}

\newtheorem{theorem}{Theorem}[section]

\theoremstyle{definition}

\theoremstyle{remark}

\numberwithin{equation}{section}

\begin{document}
\title[The Minimality of Determinantal Varieties]
{The Minimality of Determinantal Varieties}
\author[M. Bordemann, J. CHOE and J. Hoppe ]{martin bordemann, JAIGYOUNG CHOE, and JENS HOPPE}
\thanks{J.C. supported in part by NRF-2018R1A2B6004262. J.C. would like to thank the Chinese University of Hong Kong Mathematics Department for their invitation in 2019.}

\address{\parbox{\linewidth}{
Universit\'{e} de Haute Alsace, Mulhouse, France\\Korea Institute for Advanced Study, Seoul, 02455, Korea\\Braunschweig University, Germany}}
\email{martin.bordemann@uha.fr, choe@kias.re.kr, jens.r.hoppe@gmail.com}

\begin{abstract} The determinantal variety $\Sigma_{pq}$ is defined to be the set of all $p\times q$ real matrices with $p\geq q$ whose ranks are strictly smaller than $q$. It is proved that $\Sigma_{pq}$ is a minimal cone in $\mathbb R^{pq}$ and all its strata are regular minimal submanifolds.\\

\noindent{\it Keywords}: minimal submanifold, determinantal variety, helicoidal\\
{\it MSC}: 53A10, 49Q05
\end{abstract}

\maketitle

\section{introduction}
Determinantal varieties are spaces of matrices with a given upper bound on their ranks. Given $p$ and $q$ and $r<\min(p,q)$, the {\it determinantal variety} $Y_r$ is the set of all $p\times q$ matrices with ${\rm rank}\leq r$. $Y_r$ is naturally an algebraic variety as the rank condition on a matrix is equivalent to the vanishing of all of its $(r+1)\times(r+1)$ minors which are polynomials of degree $r+1$. Let $Z_r=Y_r\setminus Y_{r-1}$
denote the set of all
$p\times q$-matrices whose rank is equal to $r$. By projecting each $p\times q$ matrix of rank $r$ on
the $r$-dimensional subspace $V$ of $\mathbb{R}^p$ generated by its columns and on the $r$-dimensional subspace $W$ of $\mathbb{R}^q$ generated by its rows it can be seen that
$Z_r$ is a fibre bundle over the cartesian product of
two real Grassmannians $\mathrm{G}_r(\mathbb{R}^p)
\times\mathrm{G}_r(\mathbb{R}^q)$ where the fibre over $(V,W)$
is the set of all linear bijections $V\to W$ (canonical bundles), see e.g.~\cite[p.8]{BV88}. Hence it is a regular
submanifold of dimension $r(p-r)+qr$. It is well-known that $Y_r$ is decomposed into
the disjoint union of regular subvarieties $Z_s$,
$s$ being an integer $0\leq s\leq r$ (where of course $Y_{-1}=\emptyset$), see
e.g.~\cite[p.4-6]{BV88} for more details.

More concretely: let $M_{pq}$ the set of all real $p\times q$ matrices with $p\geq q$ and identify $M_{pq}$  with $\mathbb R^{pq}$.
 Let us define $\Sigma_{pq}=\{(x_{11},x_{12},\ldots,x_{pq})\in\mathbb R^{pq}: \det(X^TX)=0, X=(x_{ij})\in M_{pq}\}$ and call $\Sigma_{pq}$  the {\it determinantal variety} in $\mathbb R^{pq}$. It can be identified with $Y_{q-1}$. The vector space $M_{pq}$ carries a natural
 positive definite inner product
 \begin{equation}\label{EqDefInnerProductOnMatrixSpace}
   \langle X,Y\rangle = \mathrm{tr}(X^T~Y).
 \end{equation}

Many studies have been done about determinantal varieties, mostly on their algebraic properties. In this paper we are interested in an analytic property, the minimality of $\Sigma_{pq}$ in $\mathbb R^{pq}$, which we prove in 3 different ways (two of which generalize existing results for $p=q$, cp.  \cite{T}, \cite{HLT}, \cite{CH}, while concerning the first, parametric proof, see \cite{H} for some low-dimensional $p>q$ examples):
\begin{theorem}\label{Hauptsatz}
Every stratum $Z_r$ for $0\leq r\leq q-1$ of the determinantal variety $\Sigma_{pq}$ is a regular minimal submanifold of the inner product space $M_{pq}$.
\end{theorem}

\section{Parametric Proof}
\noindent Consider the real vector space $\mathbb{R}^{pr+r(q-r)}$ whose elements are written as pairs $\left(\begin{array}{c}
       a \\
       \lambda
\end{array}\right)$ where $a = (\vec{a_1} \vec{a_2} \ldots \vec{a_r}) =
(a_{Js})_{\substack
	{J = 1\ldots p \\
		s = 1\ldots r}}$ is a real $p\times r$-matrix and
$\lambda = (\lambda _{ss'})_{\substack{
		s=1\ldots r \hphantom{:=q-r} \\ s'= 1\ldots r':=q-r}}$ is a real $r\times (q-r)$-matrix.
Let $U_{pqr}$ denote the open subset of $\mathbb{R}^{pr+r(q-r)}$ consisting of those pairs
$\left(\begin{array}{c}
a \\
\lambda
\end{array}\right)$ where the matrix $a$ is of maximal rank
$r$. We denote the natural euclidean scalar product on
$U_{pqr}$ by $\gamma$, i.e.
\begin{equation}
      \gamma\left(\left(\begin{array}{c}
       a \\ \lambda
\end{array}\right),\left(\begin{array}{c}
a' \\ \lambda'
\end{array}\right)\right) = \mathrm{tr}(a^T~a')+
                \mathrm{tr}(\lambda^T~\lambda').
\end{equation}
Then the following smooth map $X:U_{pqr}\to M_{pq}\cong\mathbb{R}^{pq}$ given by
\begin{equation}\label{EqDefMapX}
X : \begin{pmatrix}
a\\
\lambda
\end{pmatrix} \mapsto X  \begin{pmatrix}
a\\
\lambda
\end{pmatrix}  = (a,a\cdot \lambda)
\end{equation}
clearly parametrizes an open subset of the stratum $Z_r$
(the regular submanifold of all rank $r$
$p\times q$-matrices) of the determinantal variety $\Sigma_{pq}$.
It can be seen as a natural chart domain of $Z_r$ (with the inverse of $X$ being the chart), the other domains
being obtained by permutations of the $r$ column vectors
$\vec{a_1}, \vec{a_2}, \ldots ,\vec{a_r}.$\\
Due to
\begin{equation}\label{EqCompDerivativeOfX}
D_{\left(\substack{ a \\ \lambda} \right)}
{\begin{pmatrix}
c\\
\mu
\end{pmatrix}}
:= \dfrac{d}{dt}X\begin{pmatrix}
a+tc\\
\lambda + t\mu
\end{pmatrix}\vert_{t=0} = (c,c\lambda+a\mu)
\end{equation}
\bigskip
the Riemannian metric $\hat{G}=X^*\langle~,~\rangle$ on $U_{pqr}$ that is induced from the inner product $\langle~,~\rangle$ (\ref{EqDefInnerProductOnMatrixSpace}) on $\mathbb R^{pq} \cong M_{pq}$ (the space of real $p\times q$ matrices) by $X$ is computed to be
\begin{equation}
\begin{array}{ll}
\widehat{G}_{\left(\substack{ a \\ \lambda} \right)}
{\left(\begin{pmatrix}
c_1\\\mu_1
\end{pmatrix},
\begin{pmatrix}
c_2\\\mu_2
\end{pmatrix}\right)} & = \mathrm{tr}(c^T_1(c_2+c_2\lambda \lambda^T + a\mu_2\lambda^T)) + \mathrm{tr}(\mu^T_1(a^Ta\mu_2+a^Tc_2\lambda))\\
& = \gamma{\left(\begin{pmatrix}
c_1\\\mu_1
\end{pmatrix},
\phi_{\left(\substack{ a \\ \lambda} \right)}
    {\begin{pmatrix}
c_2\\\mu_2
\end{pmatrix}}\right)},
\end{array}
\end{equation}
with
\begin{equation}
\phi_{\left(\substack{ a \\ \lambda} \right)}
{\begin{pmatrix}
c\\
\mu
\end{pmatrix}} = \begin{pmatrix}
c(1+\lambda \lambda^T) + a\mu \lambda^T\\
a^Tc\lambda + a^Ta\mu
\end{pmatrix}.
\end{equation}
Written as a $(pr + rr')\times (pr + rr')$ dimensional symmetric matrix,
\begin{equation}
\widehat{G} =
\begin{pmatrix}
G & B\\ B^T &D
\end{pmatrix} =
\begin{pmatrix}
g_{\beta \gamma} & b_{\beta v}\\ b_{\gamma u} & d_{uv}
\end{pmatrix}=
\begin{pmatrix}
R_{1+\lambda \lambda^T} & L_aR_{\lambda^T}\\ L_{a^T}R_\lambda & L_{a^Ta}
\end{pmatrix}
\end{equation}
where $R_Y$ resp. $L_Y$ denote right-, resp. left-, multiplication with $Y$.\\ \\
As both $G = (g_{\alpha \beta})_{\alpha, \beta = 1\ldots p\cdot r }$ and
$D =(d_{uv})_{u,v=1\ldots r\cdot r'}$
are invertible one may invert $\widehat{G}$ in the form
\begin{equation}\label{eq7}
\begin{array}{ll}
\widehat{G}^{-1} & =
\begin{pmatrix}
G^{-1}+ G^{-1}B\rho B^TG^{-1}  & -G^{-1}B\rho\\
  & \\
 -\rho B^TG^{-1} & \rho
\end{pmatrix}\\
 & \\
 & = \begin{pmatrix}
g^{\gamma\alpha}+b^\gamma_{\;\;u} \rho^{uv} b^\alpha_{\;\;v} & -b^\gamma_{\;\;u} \rho ^{uv}\\
 & \\
-\rho^{uv}b^\alpha_{\;\;v} & \rho^{uv}
\end{pmatrix}
\end{array}
\end{equation}
with $\rho$, defined as the inverse of $D-B^TG^{-1}B$, existing due to the factorization
\begin{equation}\label{eq8}
\begin{pmatrix}
G & B\\ B^T &D
\end{pmatrix} =
\begin{pmatrix}
1 & 0 \\ B^TG^{-1} & 1
\end{pmatrix}
\begin{pmatrix}
G & B \\ 0 & D-B^TG^{-1}B
\end{pmatrix};
\end{equation}
or, following from the factorization
\begin{equation}
\begin{pmatrix}
G & B\\ B^T &D
\end{pmatrix} =
\begin{pmatrix}
G' & B'\\ 0 & D'
\end{pmatrix}
\begin{pmatrix}
1 & 0 \\ K & 1
\end{pmatrix}
\end{equation}
in the form
\begin{equation}
\widehat{G}^{-1} =
\begin{pmatrix}
(G')^{-1} & -(G')^{-1}BD^{-1} \\
  & \\
 -D^{-1}B^T(G')^{-1}  & D^{-1}+D^{-1}B^T(G')^{-1}BD^{-1}
\end{pmatrix}.
\end{equation}
Due to the blocks in $\widehat{G}$ being given in terms of left- and right- multiplication one also has, more explicitly,
\begin{equation}\label{eq11}
\widehat{G}^{-1} =
\begin{pmatrix}
1-L_{P_{a}}R_{\lambda \lambda^T(1+\lambda \lambda^T)^{-1}} & -L_{a(a^{T}a)^{-1}}R_{\lambda^T} \\
  & \\
 R_{\lambda}L_{(a^Ta)^{-1}a^T}  & \rho = L_{(a^Ta)^{-1}}R_{1+\lambda^T\lambda}
\end{pmatrix}
\end{equation}
with $P_a := 1-a(a^Ta)a^T$ projecting onto the $r''=p-r$ dimensional space orthogonal to the column-space $C(a)$, i.e. the span of the $\vec{a}_s$.
Using $L_Y L_{Y'} = L_{YY'}$, $R_Y R_{Y'} = R_{Y'Y}$ one trivially verifies $\widehat{G}^{-1}\cdot \widehat{G} = 1$ also in the $R-L$ form.

To shown the minimality of $Y_r$ in the chart defined by
$X$ we have to compute the second order derivatives of
$X$ summed over the components of $\widehat{G}^{-1}$ and prove that the component normal to the tangent space of that vector (the mean curvature vector $H$) vanishes. A straight-forward, but rather lengthy computation gives
($\mu=(\alpha,u)$, $\nu=(\beta,v)$)
\begin{equation}
  \sum\hat{G}^{\mu\nu}\frac{\partial^2 X}
   {\partial x^\mu\partial x^\nu}
   \left(\begin{array}{c}
        a \\
        \lambda
   \end{array}\right)
   = -2
   D_{\left(\substack{ a \\ \lambda} \right)}X
   \left(\begin{array}{c}
   0 \\
   (a^Ta)^{-1}\lambda
   \end{array}\right),
\end{equation}
which obviously is tangential to the surface whence $H$
vanishes.\\
The computation can be shortened by the following simple
arguments:
note that the only non-vanishing second derivatives of $X$ are (independent of $s$)
\begin{equation}\label{eq12}
\dfrac{\partial^2X}{\partial a_{Js}\partial \lambda_{ss'}} = E_{Js'}
\end{equation}
(the matrix that is everywhere zero except =1 in the $J^{\text{th}}$ row and $s'$ column),
and that the $r'r'' = r'(p-r)$ normal directions are
\begin{equation}\label{eq13}
N_{s's''} := \gamma_{s'}(\lambda_{1s'}\vec{e}_{s''}, \lambda_{2s'}\vec{e}_{s''}, \ldots, \lambda_{rs'}\vec{e}_{s''}, \vec{0},\ldots, -\vec{e}_{s''},\vec{0}\ldots\vec{0} )
\end{equation}
$$
\hphantom{N_{s's''} := \gamma_{s'}(\lambda_{1s'}\vec{e}_{s''}, \lambda_{2s'}\vec{e}_{s''}, \ldots, \lambda_{rs'}\vec{e}_{s''}, \vec{0},\ldots, \ldots \ldots}  \uparrow {\scriptsize \text{at the } s'\text{-position}}
$$
where $s' = 1 \ldots r' = q-r, s'' = 1 \ldots r'' = p-r $,
and the $\lbrace\vec{e}_{s''}\rbrace_{s'' = 1\ldots r''}$ being an orthonormal basis of the kernel of $a^T$. The only non-vanishing elements of the second fundamental form are therefore
\begin{equation}\label{eq14}
h^{(t't'')}_{Js,ss'} = \gamma_{t'}\delta_{t's'}(\vec{e}_{t''})_J = h^{(t't'')}_{ss',Js} \;,
\end{equation}
hence
\begin{equation}
H^{(t't'')}\sim \gamma_{t'}\ \sum_{J,s}\widehat{G}^{Js,st'}(\vec{e}_{t''})_J
\end{equation}
for the components of the mean-curvature vector, -all of which vanish, due to the sum over $J(=1\ldots p)$, involving (only) scalar products of the $\vec{e}_{t''}$ with the $\vec{a}_s$ (all of which are $= 0$ as the $\vec{e}\,$'s are by definition in the kernel of $a^T$).
While (cp. (\ref{eq11}))
$ \widehat{G}^{Js,st'} = (-L_{a(a^Ta)^{-1}}R_{\lambda^T})^{Js,st'}$
can easily be calculated explicitely, in analogy with
\begin{equation}
B_{Jt,ss'} = \lambda_{ts'}(\vec{a}_s)_J,
\end{equation}
as $-(\vec{a}_t)_J(\vec{a}^Ta)^{-1}_{ts}\lambda_{st'}$, note that this is not really necessary, due to the following (cp.\cite{H}, p.38): as $g_{\alpha \beta} = g_{Js,Kt} \sim \delta_{JK}$ is diagonal in the indices transforming as vectors under $O(p)$, its inverse $g^{\alpha \beta}$ (as well as, trivially, $\rho^{uv}$ as the inverse of a matrix of $O(p)$-invariants) will not touch any $O(p)$ vector-index; implying that the (due to (\ref{eq12}), (\ref{eq14}) only relevant) components
\begin{equation}\label{eq17}
\widehat{G}^{Js,ss'} = -g^{Js,Kt}B_{Kt,u}\rho^{uv}
\end{equation}
must be linear combinations of the $B_{Jt,u}$, hence carry the $O(p)$ vector index entirely through the columns of $a$ (which is sufficient to conclude the vanishing of the mean curvature vector).

\section{Geometric Proof}
\noindent Our algebraic proof of the minimality of the zero determinant set $\Sigma$ in \cite{CH} was based on the property that the hypersurface $\Sigma$ is {\it helicoidal}. In this section we prove the minimality of the determinantal variety by using the fact (Theorem \ref{helicoidalthm}) that a helicoidal subset with higher codimension is still minimal (we should mention that this fact was first proved by Adrian C. Chu \cite{C}).\\

\noindent{\bf Definition.} Let $M$ be a complete Riemannian manifold and $\Sigma$ a subset of $M$. Suppose that $\Sigma_0$ is the subset of $\Sigma$ which is a twice differentiable submanifold of $M$ and $\Sigma\setminus\Sigma_0$ has measure zero on $\Sigma$. Suppose also that at any point $\bar{p}$ of $\Sigma$ there is an isometry $\varphi$ of $M$ such that
$$\varphi(\bar{p})=\bar{p}, \,\,\,\,\varphi(\Sigma)=\Sigma, \,\,\,\,\varphi_*(v)=-v\,\,\,{\rm for\,\,\,any}\,\,v\perp T_{\bar{p}}\Sigma_0.$$
Then we say that $\Sigma$ is {\it helicoidal} in $M$.

\begin{theorem}\label{helicoidalthm}
Every helicoidal subset  $\Sigma^m$ of a Riemannian manifold $M^n$ is minimal in $M$ wherever $\Sigma$ is twice differentiable.
\end{theorem}

\begin{proof}
Let $\vec{H}$ be the mean curvature vector of $\Sigma$ at a point $\bar{p}\in\Sigma_0$, that is,
$$\vec{H}=\sum_{i=1}^{m} (\overline{\nabla}_{e_i}e_i)^\perp,$$
 where $\overline{\nabla}$ is the Riemannian connection on $M$ and $e_1,\ldots,e_{m}$ are orthonormal vectors on a neighborhood of $\bar{p}$ in $\Sigma_0$. Since $\varphi(\Sigma)=\Sigma$ and $\bar{p}$ is a fixed point of $\varphi$, one sees that $\varphi_*(e_1),\ldots,\varphi_*(e_{m})$ are also orthonormal on $\Sigma$ near $\bar{p}$. Hence
 \begin{equation}\label{H}
 \varphi_*(\vec{H})=\sum_{i=1}^{m}(\overline{\nabla}_{\varphi_*(e_i)}\varphi_*(e_i))^\perp=\sum_{i=1}^{m} (\overline{\nabla}_{e_i}e_i)^\perp=\vec{H}.
 \end{equation}
 On the other hand, by the hypothesis, $\varphi_*(\vec{H})=-\vec{H}$. Hence $\vec{H}=0$ at $\bar{p}$. As $\bar{p}$ is arbitrarily chosen, one concludes that $\Sigma$ is minimal.
\end{proof}

\begin{proof} (of Theorem \ref{Hauptsatz}):   Recall the inner product $\langle\,\,,\,\rangle$ on $M_{pq}$,
$$\langle X,Y\rangle={\rm tr}(X^TY),\,\,X,Y\in M_{pq}.$$
For any $A\in O(p)$ and $X\in M_{pq}$, define $\varphi_A(X)=AX.$  Then $\varphi_A$ is an isometry on $M_{pq}$ since $\varphi_A$ is invertible and  for $Y\in M_{pq}$,
$$\langle\varphi_A(X),\varphi_A(Y)\rangle=\langle AX,AY\rangle={\rm tr}(X^TA^TAY)={\rm tr}(X^TY)=\langle X,Y\rangle.$$
Let $Z_r\subset\Sigma_{pq}$ be the set of all $p\times q$ matrices of rank $r,\,0\leq r<q$. As $\varphi_A$ preserves rank, we have
\begin{equation}\label{SS}
\varphi_A(Z_r)=Z_r.
\end{equation}

Given $X\in Z_r$, let $\mathcal{C}(X)$ be the column space of $X$ and $\mathcal{C}(X)^\perp$ its orthogonal complement in $\mathbb R^{p}$. There exists a $p\times p$ orthogonal matrix $B$ such that $\mathcal{C}(X)$ is its eigenspace with eigenvalue 1 and $\mathcal{C}(X)^\perp$ is its eigenspace  with eigenvalue $-1$. Then we have
\begin{equation}\label{XX}
\varphi_B(X)=X.
\end{equation}

Let $\mathcal{R}(X)$ be the row space of $X$. Define a subset $N$ of $M_{pq}$ by
$$N=\{Y\in M_{pq}:\mathcal{C}(Y)\subset\mathcal{C}(X)\,\,{\rm or}\,\,\mathcal{R}(Y)\subset\mathcal{R}(X)\}.$$
Choose $Y\in N$ and let $\sigma(t)$ be the curve of $p\times q$ matrices from $X$ to $Y$ defined by $\sigma(t)=X+tY,\,0\leq t\leq1$. Clearly $\sigma(t)\subset Z_r$ for sufficiently small $t$. The tangent vector $\sigma'(0)$ at $X$ equals
$$\sigma'(0)=\lim_{t\rightarrow0}\frac{\sigma(t)-X}{t}=Y\in T_XZ_r.$$
Therefore $N$ is a subset of $ T_XZ_r$ as well.
Suppose $W$ is a $p\times q$ matrix which is perpendicular to $T_XZ_r$. Then $W\perp N$.
Here we claim
$$\mathcal{C}(W)\subset\mathcal{C}(X)^\perp.$$
Let $v_i$ be the orthogonal projection of the $i$-th column of $W$ onto $\mathcal{C}(X)$ and let $W_X$ be the $p\times q$ matrix with the $i$-th column vector $v_i,\,1\leq i\leq q$. Then $W_X$ is in $ N$ and satisfies
$$\langle W_X,W_X\rangle=\sum_i|v_i|^2.$$
Since $$0=\langle W_X,W\rangle=\langle W_X,W_X\rangle=\sum_i|v_i|^2,$$
we see that $v_i=0$ for all $i$ and so the claim follows.
Hence $\varphi_B(W)=-W$. Note that  $(\varphi_B)_*=\varphi_B$ since $\varphi_B$ is linear.
Thus
\begin{equation}\label{star}
(\varphi_{B})_*(v)=-v\,\,\,{\rm for}\,\,{\rm any}\,\,v\perp T_XZ_r.
\end{equation}
It follows from $(\ref{SS}), (\ref{XX}), (\ref{star})$  that $Z_r$ is helicoidal and by Theorem \ref{helicoidalthm} it is minimal in $\mathbb R^{pq}$.
\end{proof}

\section{Level-Set Proof}
\noindent Let us also give a level-set proof, for simplicity restricting to $\Sigma_n$, the space of  all $(n+1)\times n$ matrices of  rank $(n-1)$
\begin{equation}\label{sec2eq1}
A = \begin{pmatrix}
x_1x_2\ldots x_n\\
x_{n+1} \ldots x_{2n}\\
\vdots\\
x_{n^2+1}\ldots x_{n^2+n}
\end{pmatrix}
\end{equation}
for which the first and last rows are not identically zero; $\Sigma_n$ can be characterized by the vanishing of the upper and lower $n\times n$ determinant (each of which, alone, `trivially' defining a minimal surface),
\begin{equation}\label{sec2eq2}
\chi_1:=
\begin{vmatrix}
x_1\ldots x_n\\
\vdots\\
x_{n^2-n+1}\ldots x_{n^2}
\end{vmatrix}
\stackrel{!}{=}
0
\stackrel{!}{=}
\begin{vmatrix}
x_{n+1}\ldots x_{n+n}\\
\vdots\\
x_{n^2+1}\ldots x_{n^2+n}
\end{vmatrix}
=: \chi_2 \cdot(-1)^{n-1},
\end{equation}
and the minimality-condition for the intersection ($\chi_1 = 0$ \emph{and} $\chi_2 = 0$) taking the form
\begin{equation}\label{sec2eq3}
 \forall~A\in \Sigma_n:\quad\quad\mathrm{tr}(P\cdot \partial^2\chi_{\alpha})(A)_{\alpha = 1,2} = 0
\end{equation}
with
\begin{equation}\label{sec2eq4}
P = (P_{JK}:= \delta_{JK}-\partial_J\chi_{\alpha}(M^{-1})^{\alpha \beta}\partial_K\chi_{\beta})
\end{equation}
projecting onto the tangent-space of $\Sigma_n$ if
\begin{equation}\label{sec2eq5}
M = (M_{\alpha \beta} := (\vec{\nabla}\chi_{\alpha})^T(\vec{\nabla}\chi_{\beta}))
\end{equation}
(cp.\cite{GM})
$\mathrm{tr}(\partial^2 \chi_{\alpha}) = 0$ (as all terms in $\chi_1$ and $\chi_2$ contain each of the $n^2$, resp. $n^2 + n$ variables $x_1, \ldots, x_{n^2+n}$ at most linearly), and (\ref{sec2eq3}) follows from the stronger statement (having to do with the 2 constraints already separately defining minimal surfaces)
\begin{equation}\label{sec2eq6}
\forall~A\in\Sigma_n:\quad\quad(\vec{\nabla}\chi_{\alpha})^T(\partial^2\chi_{\beta})\vec{\nabla}\chi_{\gamma}  = 0 \quad \forall \quad 1\leq \alpha, \beta, \gamma \leq 2.
\end{equation}
While for $\alpha = \gamma$ (\ref{sec2eq6}) is easy to prove,
\begin{equation}\label{sec2eq7}
(\vec{\nabla}\chi_{\alpha})^T (\partial^2\chi_{\alpha})\vec{\nabla}\chi_{\alpha} = \frac{1}{2}\chi_{\alpha}\mathrm{tr}(\partial^2\chi_{\alpha})^2
\end{equation}
\begin{equation}\label{sec2eq8}
(\vec{\nabla}\chi_{\substack{2 \\ 1}})^T (\partial^2\chi_{\substack{1 \\ 2}})\vec{\nabla}\chi_{\substack{2 \\ 1}} = \frac{1}{2}\mathrm{tr}(\partial^2\chi_1 \partial^2\chi_2)\cdot \chi_{\substack{2 \\ 1}}
\end{equation}
the corresponding identity for $\alpha \neq \gamma$ (s.b.) seemed difficult to prove. The following argument however covers \textit{all} cases : each non-zero entry  of $(\partial^2\chi_{\alpha})$ is a $(n-2)$-dimensional determinant, which occurs exactly 4 times in $\partial^2 \chi_{\alpha}$, and those 4 terms cancel in the contraction with the 2 gradients, as (on $\chi _{\alpha} = 0$).
\begin{equation}\label{sec2eq9}
\begin{array}{l}
\partial_{a+kn}\chi_1 = -\lambda_{k+1}\partial_a\chi_1 \\
\partial_{a+kn}\chi_2 = -\mu_{k+1}\partial_{a'}\chi_1 \\
{\scriptsize a = 1\ldots n, \; k = 0 \ldots n, \; a':= a+n^2}
\end{array}
\end{equation}
with $\lambda_{n+1} = 0 = \mu_1, \lambda_1 = -1 = \mu_{n+1}$, and constants $\lambda_2 \ldots \lambda_n, \mu_2 \ldots \mu_n$ appearing when writing the first row of $A$ as a linear combination of the last $n$, resp. the last row in terms of the first $n$.
As mentioned above, (\ref{sec2eq7}) and (\ref{sec2eq8}) can easily be proven (without the crucial observation (\ref{sec2eq9})), by noting (cp. \cite{HLT}, \cite{T})
\begin{equation}\label{sec2eq10}
\begin{array}{l}
\dfrac{\partial \chi_{\alpha}}{\partial A_{Ji}} = \chi_{\alpha}M^{iJ}_{\alpha} \\[0.5cm]
\dfrac{\partial^2\chi_{\alpha}}{\partial A_{Li}\partial A_{Kj}} = \chi_{\alpha}(M^{iL}_{\alpha}M^{jK}_{\alpha} - M^{jL}_{\alpha}M^{ik}_{\alpha})\\[0.5cm]
{\scriptsize J,K,L = 1\ldots n+1, \; i,j = 1\ldots n}
\end{array}
\end{equation}
where $M^{\cdot \cdot}_1$ is the inverse of $(M_1)..$, the upper $n \times n$ part of $A$ and $M^{\cdot \cdot}_2$ is the inverse of the lower one, $(M_2)..$, with the understanding that
$M^{i,J = n+1}_{1} = 0 = M^{i,J = 1}_{2}$. For $\alpha \neq \gamma$ the trick (cp.\cite{HLT}) that the $[ij]$ antisymmetry of the second derivative, which antisymmetrizes the products of the 2 gradients (hence giving the trace of 2 Hessians when $\alpha = \gamma$, with one determinant surviving) does not (simply) apply-although almost certainly resulting in
\begin{equation}\label{sec2eq11}
(\vec{\nabla}\chi_2)^T \partial^2\chi_1 \vec{\nabla}\chi_1 = \frac{\chi_2}{4} \mathrm{tr}\partial^2\chi_1\partial^2\chi_2 + \frac{\chi_1}{4}\mathrm{tr}(\partial^2\chi_2)^2,
\end{equation}
in analogy with (\ref{sec2eq7}), (\ref{sec2eq8}). Eqn (\ref{sec2eq9}) however, which in matrix-notation, due to $x_{a+(K-1)_n} = A_{Ka}$ simply reads (on $\chi_{\alpha} = 0$)
\begin{equation}\label{sec2eq12}
\begin{array}{l}
\dfrac{\partial \chi_1}{\partial A_{ka}} = -\lambda_k \dfrac{\partial \chi_1}{\partial A_{1a}} \, (= \lambda_k \lambda'_a\partial_{11} \chi_1), \\
\dfrac{\partial \chi_2}{\partial A_{ka}} = -\mu_k \dfrac{\partial \chi_2}{\partial A_{n+1,a}} \, (= \mu_k \mu'_a\partial_{n+1,n+1} \chi_2)
\end{array}
\end{equation}
can be used to easily prove (\ref{sec2eq6}), for \textit{any} $\alpha, \beta, \gamma$ as
\begin{equation}\label{sec2eq13}
\dfrac{\partial^2 \chi_{\beta}}{\partial A_{Li}\partial A_{Kj}} = -\dfrac{\partial^2 \chi_{\beta}}{\partial A_{Lj}\partial A_{Ki}},
\end{equation}
combined with (\ref{sec2eq12}), gives
\begin{equation}\label{sec2eq14}
\begin{array}{r}
\dfrac{\partial^2 \chi_{\beta}}{\partial A_{Li}\partial A_{Kj}}
(\dfrac{\partial \chi_{\alpha}}{\partial A_{Li}} \dfrac{\partial \chi_{\gamma}}{\partial A_{Kj}} +
\dfrac{\partial \chi_{\alpha}}{\partial A_{Kj}} \dfrac{\partial \chi_{\gamma}}{\partial A_{Li}} -
\dfrac{\partial \chi_{\alpha}}{\partial A_{Lj}} \dfrac{\partial \chi_{\gamma}}{\partial A_{Ki}} -
\dfrac{\partial \chi_{\alpha}}{\partial A_{Ki}} \dfrac{\partial \chi_{\gamma}}{\partial A_{Lj}}
)\\[0.5cm]
\sim (\lambda_{\alpha L} \lambda_{\gamma K} +
\lambda_{\alpha K} \lambda_{\gamma L} -
\lambda_{\alpha L} \lambda_{\gamma K} -
\lambda_{\alpha K} \lambda_{\gamma L}
) = 0,
\end{array}
\end{equation}
using that $\dfrac{\partial \chi_1}{\partial A_{1a}} = \dfrac{\partial \chi_2}{\partial A_{n+1,a}}$.\\[0.5cm]
Concerning (\ref{sec2eq12}), note that the following (stronger) statement can be easily proven for any $n \times n$ determinant $\chi$: the matrix formed out of the derivatives of $\chi$ (w.r.t the $n^2$ variables) on $\chi = 0$ is always of rank 1; choose any element (say $M_{11}$)for which the corresponding row and column are both not identically zero.
\begin{equation}\label{sec2eq15}
\chi = \sum_{a_1\ldots a_n}M_{1a_1}M_{2a_2}\ldots M_{na_n}\varepsilon^{a_1\ldots a_n}
\end{equation}
gives
\begin{equation}\label{sec2eq16}
\dfrac{\partial \chi}{\partial M_{ij}} = \sum_{a_1\ldots a_{i-1}a_{i+1} \ldots a_n}M_{1a_1}\ldots M_{i-1a_{i-1}} M_{i+1a_{i+1}} \ldots M_{na_n}\varepsilon^{a_1\ldots j \ldots a_n}_{\hphantom{a_1\ldots}\uparrow \text{at}\, i^{th} \, \text{position} } ;
\end{equation}
using
\begin{equation}\label{sec2eq17}
M_{1a_1} = \sum_{k\neq 1}\lambda_k M_{ka_1}
\end{equation}
and relabelling $a_1$ as $a_i$ after interchanging $a_1$ and $j$ in the $\varepsilon$-tensor then gives
\begin{equation}\label{sec2eq18}
\begin{array}{ll}
\dfrac{\partial \chi}{\partial M_{ij}} &
= - \lambda_i \sum M_{2a_2} \ldots M_{ia_i} \ldots M_{na_n} \varepsilon^{j^{a_2 \ldots a_n}}\\
&
= -\lambda_i \dfrac{\partial \chi}{\partial M_{1j}} \;(resp.\; \ldots = -\lambda_i \dfrac{\partial \chi_2}{\partial M_{n+1,j}} ),
\end{array}
\end{equation}
as all the terms $k \neq j$ in the sum $k \neq 1$ give zero, due to the $\varepsilon$-tensor. Analogously, then writing (on the r.h.s. of (\ref{sec2eq18})) $
\chi = \sum\limits_{c_1\ldots c_n} M_{c_1 1} M_{c_2 2} \ldots M_{c_n n} \varepsilon^{c_1 \ldots c_n}
$,
and using $M_{c_1 1} = \sum\limits_{l \neq 1} \rho_l M_{c_1 l}$ gives
\begin{equation}\label{sec2eq19}
\dfrac{\partial \chi}{\partial M_{ij}} = \lambda_i\rho_j \dfrac{\partial \chi}{\partial M_{11}}
\quad \text{on} \chi = 0.
\end{equation}
Similar arguments can be used for the general case, $p \geq q \geq r$.

\section{Generalizations}
\noindent \textbf{1.}
Among the possible generalizations, it is natural to replace real by complex numbers and to consider complex $p\times q$ matrices of rank $r$ which also forms a $r^2+r(p-r)+r(q-r)$ complex submanifold $Z^\mathbb{C}_r$ of the space of all $p\times q$-matrices,  $\mathbb{C}^{pq}$. Putting the sesquilinear
inner product
$\langle Z,Z'\rangle=
\mathrm{tr}(Z^\dagger Z')$ on $\mathbb{C}^{pq}$ we can view
it as a K\"{a}hler manifold (recall that the riemannian metric is given by the real part of the sesquilinear form) and $Z^\mathbb{C}_r$ as a complex submanifold. It is well-known that arbitrary
complex submanifolds of K\"{a}hler manifolds are always
minimal submanifolds, see e.g.~\cite[p.380]{KN69} for a simple argument. Hence $Z^\mathbb{C}_r$ is a minimal submanifold of $\mathbb{C}^{pq}$. Moreover, since
$Z^\mathbb{C}_r$ is obviously invariant under multiplication
with nonzero complex numbers, $\mathbb{C}^\times$, the above-mentioned fact also
implies that the projectivization
$Z^\mathbb{C}_r/\mathbb{C}^\times$ is a complex, hence minimal submanifold of complex projective space
$\mathbb{C}P^{pq-1}=
(\mathbb{C}^{pq}\setminus \{0\})/
\mathbb{C}^\times$(in the real case the projectivizations are minimal in spheres).

 As concrete examples let us here consider only $2$ special cases, $p = q = r+1$ and $p = 3, q = 2, r = 1$, i.e. the 8 dimensional real manifold of complex $3 \times 2$ matrices $Z =
\begin{pmatrix}
z_1 & z_4 \\
z_2 & z_5 \\
z_3 & z_6
\end{pmatrix}
= (\vec{x}+i\vec{y}, \vec{u}+i\vec{v})
\stackrel{\wedge}{=} \stackrel{\rightarrow}{z} \in \mathbb R^{12}
$
of the form
\begin{equation}\label{sec5eq1}
Z(\vec{x}, \vec{y}, \lambda, \mu) = (\vec{x}+i\vec{y}, (\lambda+i\mu)(\vec{x}+i\vec{y})),
\end{equation}
$\vec{x}, \vec{y} \in \mathbb R^{3}$ which is the
complex analog of the map $X$ is the first Section. We shall put the
inner product
$\langle Z,Z'\rangle=
\mathrm{Re}\big(\mathrm{tr}(Z^\dagger Z')\big)$ on the complex matrix space. With
\begin{equation}\label{sec5eq2}
\partial_{x^i}{\stackrel{\rightarrow}{z}} =
\begin{pmatrix}
\vec{e}_i\\
\vec{0}\\
\lambda\vec{e}_i\\
\mu\vec{e}_i
\end{pmatrix},\:
\partial_{y^i}{\stackrel{\rightarrow}{z}} =
\begin{pmatrix}
\vec{0}\\
\vec{e}_i\\
-\mu\vec{e}_i\\
\lambda\vec{e}_i
\end{pmatrix},\:
\partial_{\lambda}{\stackrel{\rightarrow}{z}} =
\begin{pmatrix}
\vec{0}\\
\vec{0}\\
\vec{x}\\
\vec{y}
\end{pmatrix},\:
\partial_{\mu}{\stackrel{\rightarrow}{z}} =
\begin{pmatrix}
\vec{0}\\
\vec{0}\\
-\vec{y}\\
\vec{x}
\end{pmatrix}
\end{equation}
one gets the following real symmetric $8\times 8$-matrix
(written in four blocks of dimension $6\times 6$, $6\times 2$,
$2\times 6$, and $2\times 2$) for the induced metric:
\begin{equation}\label{sec5eq3}
(\hat{G}_{AB}) = \left(\begin{array}{cc}
  (1+\lambda^2+\mu^2)\mathbb{1}_{6\times 6} &
    \begin{array}{cc}
      \lambda\vec{x}+\mu\vec{y} & -\lambda\vec{y}+\mu\vec{x} \\
      \lambda\vec{y}-\mu\vec{x} & \lambda\vec{x}+\mu\vec{y}
    \end{array} \\[5mm]
  \begin{array}{cc}
  \lambda\vec{x}^T+\mu\vec{y}^T & \lambda\vec{y}^T
                                   -\mu\vec{x}^T \\
  -\lambda\vec{y}^T+\mu\vec{x}^T & \lambda\vec{x}^T+\mu\vec{y}^T
  \end{array} &  (\vec{x}\,^2+\vec{y}\,^2)\mathbb{1}_{2\times2}
 \end{array}\right)
\end{equation}
(note that the two 6-dimensional vectors in the off-diagonal block(s) are orthogonal to each other, and of equal length).
As the only non-vanishing second derivatives $\partial^2_{AB}{\stackrel{\rightarrow}{z}}$ are
\begin{equation}\label{sec5eq4}
\partial^2_{x^i\lambda}{\stackrel{\rightarrow}{z}} =
\begin{pmatrix}
\vec{0}\\
\vec{0}\\
\vec{e}_i\\
\vec{0}
\end{pmatrix} = - \partial^2_{y^i\mu}{\stackrel{\rightarrow}{z}}, \quad
\partial^2_{x^i\mu}{\stackrel{\rightarrow}{z}} =
\begin{pmatrix}
\vec{0}\\
\vec{0}\\
\vec{0}\\
\vec{e}_i
\end{pmatrix} = \partial^2_{y^i\mu}{\stackrel{\rightarrow}{z}},
\end{equation}
and the 4 normal directions being of the form $\stackrel{\rightarrow}{n_{\alpha}}^T = (\ldots \ldots N^T_{\alpha})$
with $N_\alpha \in \mathbb{R}^6$, $\alpha = 1,2,3,4$, orthogonal to the plane spanned by
$
\begin{pmatrix}
\vec{x}\\ \vec{y}
\end{pmatrix} =: E_1
$ and
$
E_2 :=
\begin{pmatrix}
-\vec{y}\\ \vec{x}
\end{pmatrix}
$, the 4 mean-curvature components
$
H_{\alpha} \sim \hat{G}^{AB} \stackrel{\rightarrow}{n_{\alpha}} \cdot \;
\partial^2_{AB}{\stackrel{\rightarrow}{z}}
$
will vanish (cp. the comments around eqn (2.17), p.38 of \cite{H}) due to
$
N_{\alpha} \cdot E_1 = 0 = N_{\alpha} \cdot E_2
$;
in the notation of section 2 :
$
G = (1+\lambda^2 + \mu^2)\mathbb{1}_{6\times 6}, \;
D = ((\vec{x}\,^2+\vec{y}\,^2))\mathbb{1}_{2\times 2}
$,
\begin{equation}\label{sec5eq5}
\rho^{-1} = (\vec{x}\,^2+\vec{y}\,^2)\mathbb{1} -
\dfrac{1}{1+\lambda^2 + \mu^2}B^T B = \dfrac{\vec{x}\,^2+\vec{y}\,^2}{1+\lambda^2 + \mu^2}\mathbb{1},
\end{equation}
implying that the relevant off-diagonal part of $\hat{G}^{AB}$ is
\begin{equation}\label{sec5eq6}
-G^{-1}B\rho = -\dfrac{1}{\vec{x}\,^2+\vec{y}\,^2}(\lambda E_1 - \mu E_2,\; \lambda E_2 + \mu E_1).
\end{equation}
Concerning $p=q=r+1$, i.e. the space $\zeta_n$ of rank $(n-1)$ complex $n \times n$ matrices
\begin{equation}\label{sec5eq7}
Z =
\begin{pmatrix}
z_1 \ldots z_n \\
\vdots \\
\ldots z_{n^2}
\end{pmatrix},
\triangle := det \, Z = u+iv = 0,
\end{equation}
$x \in \mathbb{R}^{n^2}$, one can use the known fact \cite{JV} that the real and imaginary part of the determinant of a complex matrix $Z$ are `twin-harmonics', i.e. in particular satisfy
\begin{equation}\label{sec5eq8}
(\nabla u)^2 = (\nabla v)^2, \qquad \nabla u \cdot \nabla v = 0
\end{equation}
as well as
\begin{equation}\label{sec5eq9}
\dfrac{1}{u} u_i u_j u_{ij} = \dfrac{1}{v} v_i v_j v_{ij} \; (=: \rho(x)),
\end{equation}
with $\rho$ being a homogeneous polynomial of degree $2n-4$. In the simplest example, $n=2$,
\begin{equation}\label{sec5eq10}
\begin{vmatrix}
z_1 & z_2 \\ z_3 & z_4
\end{vmatrix} =
(x_1 + iy_1)(x_4 + iy_4) - (x_2 + iy_2)(x_3 + iy_3)
\end{equation}
one has
\begin{equation}\label{sec5eq11}
\nabla u =
\begin{pmatrix}
x_4\\ -x_3\\ -x_2\\ x_1\\ -y_4\\ y_3\\ y_2\\ -y_1
\end{pmatrix} \qquad
\nabla v =
\begin{pmatrix}
y_4\\ -y_3\\ -y_2\\ y_1\\ x_4\\ -x_3\\ -x_2\\ x_1
\end{pmatrix},
\end{equation}
$\rho(x) = 2$;
As, trivially, $\triangle u = 0 = \triangle v$, (\ref{sec5eq9}) implies that - separately -
\begin{equation}\label{sec5eq12}
u(x) := (x_1x_4 - x_2x_3) - (y_1y_4 - y_2y_3) \stackrel{!}{=} 0
\end{equation}
\begin{equation}\label{sec5eq13}
v(x) := (x_1y_4 + x_4y_1) - (x_2y_3 + x_3y_2) \stackrel{!}{=} 0
\end{equation}
define minimal surfaces (which in this case are simply $S^3 \times S^3$ cones, each). The interesting fact is that the \textit{intersection} of the 2 generalized Clifford cones also has zero mean curvature. This is most easily seen by differentiating (\ref{sec5eq8}), yielding
\begin{equation}\label{sec5eq14}
u_iu_{ij} = v_iv_{ij}, \; u_iv_{ij} + u_{ij}v_i = 0,
\end{equation}
hence
\begin{equation}\label{sec5eq15}
\begin{array}{l}
v_i v_{ij} u_j = u_i u_{ij} u_j = \rho u \\
u_i u_{ij} v_j = v_i v_{ij} v_j = \rho v \\
u_i v_{ij} u_j = - v_i u_{ij} u_j = - \rho v \\
v_i u_{ij} v_j = - u_i v_{ij} v_j = - \rho u ,
\end{array}
\end{equation}
all vanishing on $u = 0 = v$; using (\ref{sec2eq3}) it then immediately follows that (\ref{sec5eq7}), resp. $\zeta_n$ (for general $n$) is minimal.\\
For these K\"{a}hlerian examples it may be interesting to
consider its \emph{quantization} or noncommutative or fuzzy analog: most of all K\"{a}hler manifolds admit a quantization by Toeplitz operators or geometric quantization
(see e.g.~\cite{BMS94}) which leads to approximations
of the function algebras by a sequence of finite-dimensional
matrix algebras. Since the complex $Z^\mathbb{C}_r$ has a relatively simple structure as a complex holomorphic fibre
bundle over a cartesian product of two complex Grassmannians
it does not seem to be so hard to compute the Toeplitz quantization by representation theory of the homogeneous space. It may be simpler to do it for the projectivization
since one is then working inside complex projective space
which is compact.  In the particular example of complex rank $1$ $2\times 2$-matrices, the determinant condition gives
$Z^\mathbb{C}_r$ and its projectivization the structure of
a complex quadric for which quantization exist, see e.g.
\cite{DOP04} (concerning recent work on quantum minimal surfaces, see e.g. [12]).

\noindent  \textbf{2.} Another possible generalization is the case of \emph{pseudo-euclidean spaces}: consider again the
space $\mathbb{R}^{pq}$ of all real $p\times q$-matrices
where the positive inner product (\ref{EqDefInnerProductOnMatrixSpace}) is generalized by
the following indefinite scalar product
\begin{equation}\label{sec5eq16}
(A,B) := tr(\zeta A^T \eta B )
\end{equation}
where $\eta$ (resp. $\zeta$) is a diagonal $p \times p$ (resp. $q \times q$) matrix with $p_1 \leq p$ (resp. $q_1 \leq q$) entries $+1$ and $p_2 = p-p_1$ (resp. $q_2 = q-q_1$) entries $-1$. The signature of (\ref{sec5eq16})
is easily seen to be $(p_1p_2+q_1q_2,p_1q_2+p_2q_1)$.\\
The submanifold $Z_r\subset \mathbb{R}^{pq}$ of rank $r$ matrices is in general
no longer nondegenerate w.r.t.~(\ref{sec5eq16})
which is exemplified by the simplest non-trivial case, $p=q=2,\; r=1, \;
\zeta =
\begin{pmatrix}
1 & 0 \\ 0 & -1
\end{pmatrix},\;
\eta =
\begin{pmatrix}
1 & 0 \\ 0 & 1
\end{pmatrix},
(A,B) = ((\vec{a}_1\vec{a}_2), (\vec{b}_1\vec{b}_2)) = \vec{a}_1^{\,T} \vec{b}_1 - \vec{a}_2\vec{b}_2;
$
\begin{equation}\label{sec5eq18}
X \begin{pmatrix}
\vec{a} \\ \lambda
\end{pmatrix} :=
(\vec{a}, \lambda \vec{a}); \; \vec{a} =
\begin{pmatrix}
a^1 \\ a^2
\end{pmatrix} \in \mathbb{R}^2, \lambda \in \mathbb{R}
\end{equation}
gives $det(\hat{G}_{..}) = \vec{a}^T \vec{a} (\lambda^2-1)$,
meaning that the 3 dimensional $Z_1$ (in $\mathbb{R}^{2,2}$) defined by (\ref{sec5eq18}) contains a co-dimension $1$ $(\lambda^2 = 1 )$ singular part.\\
In order to get geometrically defined nondegenerate submanifolds consider the
following open submanifold $Z'_r$ of $Z_r$: fix
two integers $\tilde{p}_1,\tilde{q}_1$ with
$0\leq\tilde{p}_1\leq \min\{p,r\}$,
$0\leq\tilde{q}_1\leq \min\{q,r\}$,
and set
$\tilde{p}_2=r-\tilde{p}_1$, $\tilde{q}_2=r-\tilde{q}_1$.
We shall use the short notation $V_M\subset \mathbb{R}^p$
(resp.~$W_M$) to be the column space (resp.~row space) of the $p\times q$-matrix $M\in Z_r$. Moreover for
a  finite-dimensional vector space equipped with an indefinite nondegenerate symmetric bilinear form $g$ we call a vector subspace $P$ \emph{of signature
$(a,b)$} if the restriction of $g$ to $P\times P$ is nondegenerate and has signature $(a,b)$. With these notations
we set
\begin{equation}
    Z'_r=\{M\in Z_r~|~V_M~\mathrm{of~signature~}
    (\tilde{p}_1,\tilde{p}_2), W_M~\mathrm{of~signature~}
    (\tilde{q}_1,\tilde{q}_2)  \}.
\end{equation}
where we have suppressed the obvious dependence of $Z'_r$ on the choice of integers $\tilde{p}_1,\tilde{q}_1$. A lengthy computation shows that for each such choice $Z'_r$ is a
nondegenerate submanifold of $\mathbb{R}^{pq}$ equipped with
(\ref{sec5eq16}) whose induced metric has signature
\begin{equation}
 (-\tilde{p}_1\tilde{q}_1-\tilde{p}_2\tilde{q}_2
  +p_1\tilde{q}_1+p_2\tilde{q}_2+\tilde{p}_1q_1+
    \tilde{p}_1q_1,
   -\tilde{p}_1\tilde{q}_2-\tilde{p}_2\tilde{q}_1
   +p_1\tilde{q}_2+p_2\tilde{q}_1+\tilde{p}_1q_2+
   \tilde{p}_2q_1 ).
\end{equation}
As a drawback of the above indefinite scalar product
(\ref{sec5eq16}) we can deduce -using the preceding equation-- that a Minkowski
signature $(1,s)$ can only be obtained for rather trivial cases.
A generalization of the techniques described in the previous
Sections, in particular Section 3, shows that each
submanifold $Z'_r$ is helicoidal and hence minimal.

\end{document}